\documentclass[12pt,a4paper,reqno]{amsart}
\usepackage{amsmath}
\usepackage{amsfonts}
\usepackage{amssymb}
\usepackage{amscd}
\usepackage{bm} 

\newcommand\eps{\varepsilon}
\newcommand\R{{\mathbf{R}}}
\newcommand\C{{\mathbf{C}}}

\newcommand\rad{{\operatorname{rad}}}
\newcommand\loc{{\operatorname{loc}}}

\theoremstyle{plain}
  \newtheorem{theorem}[subsection]{Theorem}
  
  \newtheorem{proposition}[subsection]{Proposition}
  
  \newtheorem{corollary}[subsection]{Corollary}

\theoremstyle{remark}
  \newtheorem{remark}[subsection]{Remark}
  \newtheorem{example}[subsection]{Example}

\theoremstyle{definition}
  \newtheorem{definition}[subsection]{Definition}

\include{psfig}

\begin{document}

\title[Weak solutions of mass-critical NLS]{Global existence and uniqueness results for weak solutions of the focusing mass-critical non-linear Schr\"odinger equation}
\author{Terence Tao}
\address{Department of Mathematics, UCLA, Los Angeles CA 90095-1555}
\email{tao@math.ucla.edu}

\vspace{-0.3in}
\begin{abstract}
We consider the focusing mass-critical NLS $iu_t + \Delta u = - |u|^{4/d} u$ in high dimensions $d \geq 4$, with initial data $u(0) = u_0$ having finite mass $M(u_0) = \int_{\R^d} |u_0(x)|^2\ dx < \infty$.  It is well known that this problem admits unique (but not global) strong solutions in the Strichartz class $C^0_{t,\loc} L^2_x \cap L^2_{t,\loc} L^{2d/(d-2)}_x$, and also admits global (but not unique) weak solutions in $L^\infty_t L^2_x$.  In this paper we introduce an intermediate class of solution, which we call a \emph{semi-Strichartz class solution}, for which one does have global existence and uniqueness in dimensions $d \geq 4$.  In dimensions $d \geq 5$ and assuming spherical symmetry, we also show the equivalence of the Strichartz class and the strong solution class (and also of the semi-Strichartz class and the semi-strong solution class), thus establishing ``unconditional'' uniqueness results in the strong and semi-strong classes.  With these assumptions we also characterise these solutions in terms of the continuity properties of the mass function $t \mapsto M(u(t))$.
\end{abstract}

\maketitle

\section{Introduction}

\subsection{The focusing mass-critical NLS}

This paper is concerned with low regularity solutions $u: I \times \R^d \to \C$ to the initial value problem to the focusing mass-critical nonlinear Schr\"odinger equation (NLS)
\begin{equation}\label{nls}
\begin{split}
i u_t + \Delta u &= F(u)\\
u(t_0) &= u_0
\end{split}
\end{equation}
in high dimensions $d \geq 4$, where $I \subset \R$ is a time interval containing a time $t_0 \in \R$, $F: \C \to \C$ is the nonlinearity $F(z) := - |z|^{4/d} z$, and we assume $u_0$ to merely lie in the class $L^2_x(\R^d)$ of functions of finite mass $M(u_0) := \int_{\R^2} |u_0(x)|^2\ dx$.  The exponent $1+\frac{4}{d}$ in the nonlinearity makes the equation \emph{mass-critical}, so that the mass $M(u)$ is invariant under the scaling $u(t,x) \mapsto \frac{1}{\lambda^{d/2}} u(\frac{t}{\lambda^2}, \frac{x}{\lambda})$ of the equation.  The mass is also \emph{formally} conserved by the flow, thus we formally have $M(u(t)) = M(u_0)$ for all $t$, though it will be important in this paper to bear in mind that this formal mass conservation can break down if the solution is too weak in nature.

\begin{remark}
The condition $d \geq 4$ is assumed in order to ensure that the nonlinearity $F(u)$ is locally integrable in space for $u \in L^2_x(\R^d)$, so that the equation \eqref{nls} makes sense distributionally\footnote{Here and in the sequel, we use the subscript $\loc$ to denote boundedness of norms on compact sets, thus for instance $u \in L^\infty_{t,\loc} L^2_x(I \times \R^d)$ if and only if $u \in L^\infty_{t} L^2_x(J \times \R^d)$ for all compact $J \subset I$, with the function space $L^\infty_{t,\loc} L^2_x(I \times \R^d)$ then being given the induced Frechet space topology.} for $u \in L^\infty_{t,\loc} L^2_x(I \times \R^d)$.  It will be clear from our arguments that our results would also apply if $F$ were replaced by any other nonlinearity of growth $1+\frac{4}{d}$, whose derivative grew like $|z|^{4/d}$ and which enjoyed the Galilean invariance $F(e^{i\theta} z) = e^{i\theta} F(z)$ (in order to formally conserve mass), though in this more general setting, the mass $M(Q)$ of the ground state would need to be replaced by some unspecified constant $\eps_{F,d} > 0$ depending on the nonlinearity $F$ and the dimension $d$.
\end{remark}

The notion of a distributional solution, by itself, is too weak for applications; for instance, one has difficulty interpreting what the initial data condition $u(0) = u_0$ means for a distributional soution in $L^\infty_{t,\loc} L^2_x$.  In practice, one strengthens the notion of solution at this regularity by working with the integral formulation\footnote{We adopt the usual convention $\int_b^a = - \int_a^b$ when $a < b$.}
\begin{equation}\label{integral}
u(t) = e^{i(t-t_0)\Delta} u_0 + i \int_{t_0}^t e^{i(t-t')\Delta} F(u(t'))\ dt'
\end{equation}
of the equation, where $e^{it\Delta}$ is defined via the Fourier transform $\hat u(\xi) := \int_{\R^d} e^{-i x \cdot \xi} u(x)\ dx$ as
$$ \widehat{e^{it\Delta} u}(\xi) := e^{-it|\xi|^2} \hat u(\xi)$$
which is well-defined in the class of tempered distributions.  

\begin{remark}\label{weakconv} If $u_0 \in L^2_x(\R^d)$ and $u \in L^\infty_{t,\loc} L^2_x(I \times \R^d)$, then $F(u) \in L^\infty_{t,\loc} L^1_{x,\loc}(I \times \R^d)$, and the right-hand side of \eqref{integral} makes sense as a tempered distribution in $x$ for each time $t$.  Furthermore, it is easy to verify (by the standard duality argument) that the right-hand side of \eqref{integral} is continuous in $t$ in the topology ${\mathcal S}(\R^d)^*$ of tempered distributions.
\end{remark}

\subsection{Weak, strong, and Strichartz class solutions}

With these preparations, we can now introduce the three standard solution classes for this problem in $L^2_x(\R^d)$.

\begin{definition}[Weak, strong, Strichartz solutions]  Fix a dimension $d \geq 4$, an initial data $u_0 \in L^2_x(\R^d)$ and a time interval $I \subset \R$ containing a time $t_0 \in \R$.
\begin{itemize}
\item A \emph{weak solution} (or \emph{mild solution}) to \eqref{nls} is a function $u \in L^\infty_{t,\loc} L^2_x(I \times \R^d)$ which obeys \eqref{integral} in the sense of tempered distributions for almost every\footnote{By definition of $L^\infty_t$, weak solutions are only defined for almost every time $t$, though from Remark \ref{weakconv}, one can canonically define $u(t)$ for all $t \in I$.} time $t$. \item A \emph{strong solution} to \eqref{nls} is a weak solution $u$ such that $t \mapsto u(t)$ is continuous in the $L^2_x$ topology, thus $u$ lies in $C^0_{t,\loc} L^2_x(I \times \R^d)$.
\item A \emph{Strichartz-class solution} to \eqref{nls} is a strong solution which also lies in $L^2_{t,\loc} L^{2d/(d-2)}_x(I \times \R^d)$, thus $u$ lies in $C^0_{t,\loc} L^2_x(I \times \R^d) \cap L^2_{t,\loc} L^{2d/(d-2)}_x(I \times \R^d)$.
\end{itemize}
\end{definition}

\begin{remark}[Shifting initial data]  Because the right-hand side of \eqref{integral} is continuous in the distributional topology for any of the above three notions of solutions, we observe that if $u$ is a solution to \eqref{nls} in any of the above classes on an interval $I$, and $t_1 \in I$, then $u$ is also a solution to \eqref{nls} in the same class with initial time $t_1$ and initial data $u(t_1)$ (as defined using the right-hand side of \eqref{integral}).  Thus one may legitimately discuss solutions to NLS in one of the above three classes without reference to an initial time or initial data.
\end{remark}

\begin{remark}\label{trivial}  For future reference, we make the trivial remark that if one restricts a solution in any of the above classes to a sub-interval $J \subset I$, then one still obtains a solution in the same class.  Conversely, if one has a family of solutions in the same class on different time intervals $I_n$, such that $\bigcap_n I_n \neq \emptyset$ and any two solutions agree on their common domain of definition, then one can glue them together to form a solution in the same class on the union $\bigcup_n I_n$.  
\end{remark}

\begin{remark}\label{weak2}  From Remark \ref{weakconv} we make the important observation that if $u \in L^\infty_t L^2_x(I \times \R^d)$ is a weak solution to \eqref{nls}, then the map $t \mapsto u(t)$ is continuous in the weak topology of $L^2_x(\R^d)$.  In particular we have the convergence property
\begin{equation}\label{wlim}
\lim_{t' \to t} \langle u(t'), u(t) \rangle_{L^2_x(\R^d)} = M(u(t))
\end{equation}
for all $t \in I$, which by the cosine rule implies the asymptotic mass decoupling identity
\begin{equation}\label{cosine}
\lim_{t' \to t} M(u(t')) - M(u(t')-u(t)) - M(u(t)) = 0.
\end{equation}
Thus any $L^2$ discontinuity of $u$ at $t$ can be detected and quantified by the mass function $t \mapsto M(u(t))$; in particular, the solution $t \mapsto u(t)$ is continuous in $L^2$ at precisely those points for which the mass function $t \mapsto M(u(t))$ is continuous.
\end{remark}

In the Strichartz class, one has a satisfactory local existence and uniqueness theory:

\begin{proposition}[Local existence and uniqueness in the Strichartz class]\label{strichprop}  Let $d \geq 4$, $u_0 \in L^2_x(\R^d)$, and $t_0 \in \R$.
\begin{itemize}
\item[(i)] (Local existence) There exists an open interval $I$ containing $t_0$ and a Strichartz class solution $u \in C^0_{t,\loc} L^2_x(I \times \R^2) \cap L^2_{t,\loc} L^{2d/(d-2)}_x(I \times \R^d)$.
\item[(ii)] (Uniqueness) If $I$ is an interval containing $0$, and $u, u' \in C^0_{t,\loc} L^2_x(I \times \R^2) \cap L^2_{t,\loc} L^{2d/(d-2)}_x(I \times \R^d)$ are Strichartz class solutions to \eqref{nls} on $I$, then $u = u'$.
\item[(iii)] (Mass conservation) If $u \in C^0_{t,\loc} L^2_x(I \times \R^2) \cap L^2_{t,\loc} L^{2d/(d-2)}_x(I \times \R^d)$ is a Strichartz solution, then the function $t \mapsto M(u(t))$ is constant.
\end{itemize}
\end{proposition}

\begin{proof} This is a standard consequence of the endpoint Strichartz estimate\footnote{Here and in the sequel we use the usual notation $X \lesssim Y$ or $X=O(Y)$ to denote the estimate $|X| \leq CY$ for some absolute constant $C > 0$; if the implied constant $C$ depends on a parameter (such as $d$), we will indicate this by subscripts, e.g. $X \lesssim_d Y$ or $X = O_d(Y)$.}
\begin{equation}\label{endpoint-strichartz}
\| u \|_{L^2_t L^{2d/(d-2)}_x(I \times \R^d)} +
\| u \|_{C^0_t L^2_x(I \times \R^d)} \lesssim_d \|u(t_0)\|_{L^2_x(\R^d)} + \| i u_t + \Delta u \|_{L^2_t L^{2d/(d+2)}_x(I \times \R^d)}
\end{equation}
from \cite{tao:keel}: see \cite{cwI}, \cite{caz}.  Mass conservation is obtained in these references by first regularising the data and nonlinearity so that the solution is smooth (and the formal conservation of mass can be rigorously justified), and then taking limits using \eqref{endpoint-strichartz}.
\end{proof}

Because of this proposition (and Remark \ref{trivial}), every initial data $u_0 \in L^2_x(\R^d)$ and initial time $t_0 \in \R$ admits a unique \emph{maximal Strichartz-class Cauchy development} $u \in C^0_{t,\loc} L^2_x(I \times \R^2) \cap L^2_{t,\loc} L^{2d/(d-2)}_x(I \times \R^d)$ where $I$ is an open interval containing $t_0$, and $u$ is a Strichartz-class solution to \eqref{nls} which cannot be extended to any larger time interval.  

Unfortunately, the lifespan $I$ of this maximal Strichartz-class Cauchy development need not be global if the mass $M(u_0)$ is large.  For instance, if $Q$ is a non-trivial Schwartz-class solution to the ground state equation 
\begin{equation}\label{gse}
\Delta Q + |Q|^{4/d} Q = Q,
\end{equation}
then as is well known, the function
\begin{equation}\label{pc}
u(t,x) := \frac{1}{|t|^{d/2}} e^{-i/t} e^{i|x|^2/4t} Q(x/t)
\end{equation}
is a Strichartz-class solution on $(0,+\infty) \times \R^d$ or $(-\infty,0) \times \R^d$ but cannot be extended in this class across the time $t=0$.  One can also use Glassey's virial identity \cite{glassey} to infer indirectly the non-global nature of maximal Strichartz-class Cauchy developments for suitably smooth and decaying data with negative energy.

\begin{remark}  In the defocusing case $F(z) = + |z|^{4/d} z$ it is conjectured that all maximal Strichartz-class Cauchy developments are global.  This has recently been established in the spherically symmetric case in \cite{tvz-higher}, and is also known for data with additional regularity (e.g. energy class) or decay (e.g. $xu_0 \in L^2_x(\R^d)$), or with small mass; see \cite{tvz-higher} and the references therein for further discussion.  In the focusing case, the results of \cite{kvz} give global existence for spherically symmetric data when the mass $M(u_0)$ is strictly less than the mass $M(Q)$ of the ground state; see \cite{visan-endpoint} for a treatment of the endpoint case $M(u_0)=M(Q)$.    Again, it is conjectured that the same results hold without the spherical symmetry assumption, but this remains open.
\end{remark}

On the other hand, it is possible to continue solutions in a weak sense beyond the time for which Strichartz-class solutions blow up.  In particular, we have the following standard result:

\begin{proposition}[Global existence in the weak class]\label{wk}  Let $d \geq 4$, $u_0 \in L^2_x(\R^d)$, and $t_0 \in \R$.  Then there exists a global weak solution $u \in L^\infty_t L^2_x(\R \times \R^d)$ to \eqref{nls}.  Furthermore we have $M(u(t)) \leq M(u_0)$ for all $t \in \R$.
\end{proposition}

\begin{proof} We will prove a stronger result than this shortly, so we only give a sketch of proof here. By Remark \ref{trivial} and time reversal symmetry, it suffices to build a solution on $[t_0,+\infty)$.  For each $\eps > 0$, one can easily use parabolic theory to construct a global (strong) solution to the damped NLS $iu^{(\eps)}_t + \Delta u^{(\eps)} = i\eps \Delta u^{(\eps)} + F_\eps(u^{(\eps)})$ on $[t_0,+\infty)$, whose mass is bounded above by $M(u_0)$, where $F_\eps$ is a suitably damped version of $F$ (e.g. $F_\eps(z) := -\max(|z|,1/\eps)^{4/d} z$); extracting a weakly convergent subsequence and taking weak limits we obtain the claim.
\end{proof}

Unfortunately, while these weak solutions are global, they are non-unique, as the following standard example shows.

\begin{example}\label{nonunique}  Consider the function given by \eqref{pc} for $t \in (0,+\infty)$ and by zero for $t \in (-\infty,0]$.  This is a global weak solution in the sense of the above proposition (taking $t_0$ to be any positive time, and setting $u_0 = u(t_0)$), but is not unique; if one for instance takes $u$ to equal \eqref{pc} for $t \in (-\infty,0)$ rather than equal to zero, then the new solution is still a global weak solution with the same initial data.  Note that a modification of this example shows that uniqueness of weak solutions can break down even if the initial data is zero, and so one cannot hope to recover uniqueness purely by strengthening the hypotheses on the initial data.
\end{example}

\begin{remark} Example \ref{nonunique} also shows that mass is not necessarily conserved for weak solutions.  On the other hand, from \eqref{cosine} we see that the function $t \mapsto M(u(t))$ is lower semi-continuous, at least.
\end{remark}

\subsection{Semi-Strichartz solutions}

To summarise the discussion so far, the Strichartz class of solutions has uniqueness but no global existence, while the class of weak solutions has global existence but no uniqueness.  It is thus natural to ask whether there is an intermediate class of solutions for which one has both global existence and uniqueness.  To answer this we define some further solution classes.

\begin{definition}[Semi-strong and semi-Strichartz solutions]\label{semidef}  Fix a dimension $d \geq 4$, an initial data $u_0 \in L^2_x(\R^d)$ and a time interval $I \subset \R$ containing a time $t_0 \in \R$.  A \emph{semi-strong solution} (resp. \emph{semi-Strichartz class solution}) to \eqref{nls} is a weak solution $u$ such that for every $t \in I \cap [t_0,+\infty)$ there exists $\eps > 0$ such that $u$ is a strong solution (resp. Strichartz class solution) when restricted to $I \cap [t,t+\eps)$, and for every $t \in I \cap (-\infty,t_0]$ there exists $\eps > 0$ such that $u$ is a strong solution (resp. Strichartz class solution) when restricted to $I \cap (t-\eps,t]$.
\end{definition}

We summarise the obvious inclusions between the five classes of solution in Figure \ref{include}.  Note that unlike the weak, strong, and Strichartz classes, the semi-strong and semi-Strichartz classes of solution depend on the choice of initial time $t_0$.

\begin{figure}
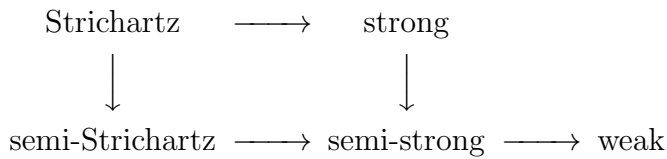
\label{include}
$$
\begin{CD}
\operatorname{Strichartz}       @>>>        \operatorname{strong}    @. \\
@VVV                                         @VVV                             @.       \\
\operatorname{semi-Strichartz}  @>>>        \operatorname{semi-strong}  @>>>   \operatorname{weak}
\end{CD}
$$
\caption{Inclusions between solution classes.  In dimensions $d \geq 5$ and assuming spherical symmetry, we will show that two horizontal inclusions on the left are in fact equivalences.}
\end{figure}

\begin{example}\label{semiex}  Consider the weak solution $u \in L^\infty_t L^2_x(\R \times \R^d)$ which is given by \eqref{pc} for $t > 0$ and is zero for $t \leq 0$, let $t_0 > 0$, and set $u_0 := u(t_0)$.  Then $u$ is a semi-Strichartz class solution (and thus semi-strong solution) to \eqref{nls}, but is not strong or Strichartz-class.  If one redefines $u$ for $t < 0$ by \eqref{pc}, then $u$ remains a weak solution, but is no longer semi-strong or semi-Strichartz.
\end{example}

\begin{remark} The constructions in \cite{bowa}, in our notation, yield semi-Strichartz class solutions which blow up in the Strichartz class at a specified finite set of points in time, and are equal to a prescribed state in $L^2_x(\R^d)$ at the final blowup time, in dimensions $d=1,2$.
\end{remark}

Our first main result is that the semi-Strichartz solution class enjoys global existence and uniqueness:

\begin{theorem}[Global existence and uniqueness in the semi-Strichartz solution class]\label{main}  Let $d \geq 4$, $u_0 \in L^2_x(\R^d)$, and $t_0 \in \R$.    Then there exists a global semi-Strichartz class solution $u \in L^\infty_t L^2_x(\R \times \R^d)$ to \eqref{nls}.  Furthermore, this solution is unique in the sense that any other semi-Strichartz solution to \eqref{nls} on a time interval $I$ containing $t_0$ is the restriction of $u$ to $I$.  Finally, $M(u(t))$ is monotone non-increasing for $t \geq t_0$ and monotone non-decreasing for $t \leq t_0$ (in particular, the only possible discontinuities are jump discontinuities), and has a jump discontinuity exactly at those times $t$ for which $u$ is not locally a Strichartz class solution.
\end{theorem}

\begin{remark}  Informally, the unique semi-Strichartz class solution is formed by solving the equation in the Strichartz class whenever possible, and deleting any mass that escapes to spatial or frequency infinity when the solution leaves the Strichartz class.  The relationship between this class of solution and Strichartz class solutions is analogous to the relationship between Ricci flow with surgery and Ricci flow in the work of Perelman \cite{per1}, \cite{per2}, though of course the situation here is massively simpler than with Ricci flow due to the semilinear and flat nature of our equation.  On the other hand, the ``entropy'' type solutions constructed in Proposition \ref{wk} do not necessarily converge to the solution in Theorem \ref{main}.  For instance, the arguments in \cite{merle} can be adapted to show that if one starts with the initial data of \eqref{pc} at time $t=-1$ (say) and evolves a parabolically regularised version of \eqref{nls} using some viscosity parameter $\eps$, then the solution at $t=+1$ can converge to an arbitrary phase rotation of the solution \eqref{pc} along a subsequence of $\eps$, and in particular these solutions do not converge to the semi-Strichartz solution (which vanishes after the singularity time).  However, it is conceivable that the entropy solutions do converge to the semi-Strichartz solutions for \emph{generic} data, although the author does not know how one would try to prove this.
\end{remark}

\begin{remark} One can push the global existence result further, to obtain scattering at $t = \pm \infty$, and can in fact even push the solution ``beyond'' $t=+\infty$ and $t=-\infty$ by using the pseudoconformal transform or lens transform, in the spirit of \cite{tao-lens}.  We omit the details.
\end{remark}

\begin{remark} While the semi-Strichartz class enjoys global existence and uniqueness, it does not enjoy continuous dependence on the data and is thus not a well-posed class of solutions.  Indeed, if one considers the solution in Example \ref{semiex} for the spherically symmetric ground state $Q$, and then perturbs the initial data $u_0=u(t_0)$ to have slightly smaller mass (while staying spherically symmetric), then from the results in \cite{kvz} we know that the perturbed solution exists globally in the Strichartz class, and in particular has mass close to $M(Q)$ for all negative times, in contrast to the original solution in Example \ref{semiex} which has zero mass for all negative times, thus contradicting continuous dependence on the data in any reasonable topology.  Indeed this argument strongly suggests that there is no solution class for this equation which is globally well-posed in the sense that one simultaneously has global existence, uniqueness, and continuous dependence of the data, and which is compatible with the Strichartz class of solutions.
\end{remark}

\begin{remark} In \cite{mer2}, \cite{mer3}, solutions to \eqref{nls} are constructed which are initially in $H^1_x(\R^d)$, but at the first blowup time develop a single point of concentration, plus a residual component $u^*$ which is not in $L^p_x(\R^d)$ for any $p>2$, and in particular has left $H^1_x(\R^d)$.  The semi-Strichartz solution would continue the evolution from $u^*$ at this time.  Thus, we do not have persistence of regularity for the semi-Strichartz class: a semi-Strichartz solution can exit the space in finite time.  (A similar phenomenon for the supercritical focusing NLS was also obtained in \cite{mer}.  In contrast, the solution in Example \ref{semiex} has $H^1$ norm going to infinity as $t \to 0^+$, but never actually leaves $H^1_x(\R^d)$; similarly for the solutions in \cite{bowa}).  
\end{remark}

Theorem \ref{main} is in fact an easy consequence of Proposition \ref{strichprop} and is proven in Section \ref{mainproof}.  One can be somewhat more precise about the jump discontinuities:

\begin{theorem}[Quantisation of mass loss]\label{quant} Let $d \geq 4$, $u_0 \in L^2_x(\R^d)$, and $t_0 \in \R$.    Let $u \in L^\infty_t L^2_x(\R \times \R^d)$ be the unique global semi-Strichartz class solution to \eqref{nls} given by Theorem \ref{main}.  Then there exists an absolute constant $\eps_d > 0$ (depending only on $d$) such that every jump discontinuity of the function $t \mapsto M(u(t))$ has jump at least $\eps_d$.  If $u_0$ is spherically symmetric, one can take $\eps_d$ to be the mass $M(Q)$ of the ground state.
\end{theorem}

\begin{remark} A closely related result in the spherically symmetric case was established in \cite[Corollary 1.12]{ktv}, in which it was shown that any blowup of a spherically symmetric Strichartz class solution in two dimensions must concentrate an amount of mass at least equal to the ground state $M(Q)$; the same result in higher dimensions follows by the same argument together with the results in \cite{kvz}.  Indeed, we will use the results in \cite{kvz} to establish the spherically symmetric case of this theorem.  From Example \ref{nonunique} we see that $M(Q)$ cannot be replaced by any larger quantity in the above theorem.

Theorem \ref{quant} is of course consistent to the existence of a lower bound $\eps_d$ for mass concentration at a point, see \cite{borg:book}, \cite{mv}, \cite{keraani}, although neither result seems to directly imply the other.  (The proof of Theorem \ref{quant} uses global-in-space Strichartz estimates, whereas the mass concentration result requires more localized tools.)
\end{remark}

Theorem \ref{quant}, combined with Theorem \ref{main} and Proposition \ref{strichprop}(iii), has an immediate corollary:

\begin{corollary}\label{quantcor}  If $u$ is a global semi-Strichartz class solution to \eqref{nls}, then the function $t \mapsto M(u(t))$ is piecewise constant with at most finitely many jump discontinuities, with $u$ being a Strichartz class solution on each of the piecewise constant intervals.
\end{corollary}

We prove Theorem \ref{quant} in Section \ref{quant-sec}.

\subsection{The spherically symmetric case}

Now we turn to the question of whether strong (resp. semi-strong) solutions are necessarily in the Strichartz class (resp. semi-Strichartz class), which would imply (by Proposition \ref{strichprop} and Theorem \ref{main}) that they are unique.  These type of results are known as \emph{unconditional uniqueness} (or \emph{unconditional well-posedness}) results in the literature.  For solutions in higher regularities, such as the energy class, one can obtain unconditional uniqueness by exploiting Sobolev embedding to obtain additional integrability of the strong solution $u$; see \cite{katounique}, \cite{katounique2}, \cite{furioli}, \cite{furioli2}, \cite{caz}, \cite{tsutsumi}.  Unfortunately at the $L^2_x(\R^d)$ level of regularity, for which Sobolev embedding is not available\footnote{Related to this difficulty is the Galilean invariance of the NLS equation at $L^2_x(\R^d)$, which strongly suggests that direct application of Sobolev or Littlewood-Paley theory is unlikely to be helpful.}, it appears to be rather difficult to establish such an unconditional uniqueness result, although the author tentatively conjectures it to be true.  On the other hand, we were able to establish this uniqueness under the additional simplifying assumption of spherical symmetry (and assuming very high dimension $d \geq 5$), thus replacing the data space $L^2_x(\R^d)$ by the subspace $L^2_\rad(\R^d)$ of spherically symmetric functions:

\begin{theorem}[Unconditional uniqueness for spherically symmetric solutions]\label{uncond}  Let $d \geq 5$, $u_0 \in L^2_\rad(\R^d)$, $I$ be an interval, and $t_0 \in \R$.  Let $u \in L^\infty_t L^2_x(I \times \R^d)$ be a spherically symmetric weak solution to \eqref{nls}.  Then the following are equivalent:
\begin{itemize}
\item[(i)] $u$ is a Strichartz class solution.
\item[(ii)] $u$ is a strong solution.
\item[(iii)] The function $t \mapsto M(u(t))$ is constant.
\item[(iv)] The function $t \mapsto M(u(t))$ is continuous.
\item[(v)] One has $M(u(t)) \geq \limsup_{t' \to t} M(u(t')) - \eps_d$ for all $t \in I$, where $\eps_d > 0$ is a suitably small absolute constant depending only on $d$.  (Note from lower semi-continuity that we automatically have $M(u(t)) \leq \limsup_{t' \to t} M(u(t'))$.)
\end{itemize}
\end{theorem}

\begin{example} If $u$ is given by \eqref{pc} for $t \neq 0$ and a spherically symmetric $Q$ and vanishes for $t=0$, then $u$ is a spherically symmetric weak solution but fails to conserve mass at $t=0$, and is thus not in the Strichartz class in a neighbourhood of $t=0$.
\end{example}

\begin{remark}
From Theorem \ref{uncond} and Proposition \ref{strichprop}(ii), we see that spherically symmetric strong solutions to \eqref{nls} are unique.  Another quick corollary of Theorem \ref{uncond} is that any spherically symmetric weak solution whose mass is always strictly smaller than $\eps_d$ is necessarily a Strichartz class solution (and hence strong solution also), and thus also unique.  In view of Theorem \ref{quant}, it is natural to conjecture that one can take $\eps_d$ to be the mass $M(Q)$ of the ground state, which is the limit of weak uniqueness thanks to Example \ref{nonunique}, but our methods do not give this.
\end{remark}

\begin{remark} The above theorem shows that if a weak solution fails to be in the Strichartz class, then at some time $t$ it must lose at least a fixed amount $\eps_d$ of mass, though it is possible that this mass is then instantly recovered (consider for instance the solution given by \eqref{pc} for $t \neq 0$ and zero for $t=0$).  On the other hand, it is conceivable that there exist weak solutions in which the mass function $t \mapsto M(u(t))$ exhibits oscillating singularities rather than jump discontinuities, in which the mass oscillates infinitely often as one approaches a given time; the above theorem implies that the net oscillation is at least $\eps_d$ but does not otherwise control the behaviour of this function.  If for instance there existed a non-trivial weak solution on a compact interval $I$ which vanished at both endpoints of the interval (cf. \cite{SchefferNonunique}), then one could achieve such an oscillating behaviour by gluing together rescaled, time-translated versions of this solution.  However, we do not know if such weak solutions exist; solutions such as \eqref{pc} constructed using the pseudo-conformal transformation only exhibit vanishing at a single time $t$.
\end{remark}

\begin{remark} Note that we need to assume the solution is spherically symmetric, and not just the initial data.  In the category of weak solutions, at least, it is not necessarily the case that spherically symmetric data leads to spherically symmetric solutions: consider for instance the weak solution which is equal to a time-translated version of \eqref{pc} for $t \neq 0$ and vanishes for $t=0$; this solution is spherically symmetric at time zero but not at other times.
\end{remark}

We prove Theorem \ref{uncond} in Section \ref{uncond-sec}, after establishing an important preliminary smoothing estimate for weak solutions in Section \ref{weak-sec}.  Our main tool here is the in/out decomposition of waves used in \cite{tao:radialfocus}, \cite{ktv}, which is particularly powerful for understanding the dispersion of spherically symmetric waves, and which we will rely upon heavily in order to establish a substantial gain of regularity for weak solutions.  Our arguments only barely fail at $d=4$ and it is quite likely that a refinement of the methods here can be extended to that case, but we do not pursue this matter here.

There is an analogue of Theorem \ref{uncond} for semi-strong and semi-Strichartz class solutions:

\begin{theorem}[Characterisation of spherically symmetric semi-Strichartz solution]\label{semis} Let $d \geq 5$, $u_0 \in L^2_\rad(\R^d)$, $I$ be an interval, and $t_0 \in \R$.  Let $u \in L^\infty_t L^2_x(I \times \R^d)$ be a spherically symmetric weak solution to \eqref{nls}.  Then the following are equivalent:
\begin{itemize}
\item[(i)] $u$ is the unique semi-Strichartz solution given by Theorem \ref{main} (restricted to $I$, of course).
\item[(ii)] $u$ is a semi-strong solution.
\item[(iii)] The function $t \mapsto M(u(t))$ is right-continuous for $t \geq t_0$ and left-continuous for $t \leq t_0$, and is piecewise constant with only finitely many jump discontinuities, with each jump being at least $M(Q)$ in size.
\item[(iv)] The function $t \mapsto M(u(t))$ is right-continuous for $t \geq t_0$ and left-continuous for $t \leq t_0$.
\item[(v)] One has $M(u(t)) \geq \limsup_{t' \to t^+} M(u(t')) - \eps_d$ for all $t \geq t_0$ and $M(u(t)) \geq \limsup_{t' \to t^-} M(u(t')) - \eps_d$ for all $t \leq t_0$, where $\eps_d > 0$ is a suitably small absolute constant depending only on $d$. 
\end{itemize}
\end{theorem}

We prove Theorem \ref{semis} in Section \ref{semis-sec}, using a minor modification of the argument used to prove Theorem \ref{uncond}.

The author thanks Jim Colliander for helpful discussions, and Tim Candy, Fabrice Planchon and Monica Visan for corrections and references, and the anonymous referee for valuable comments.  The author is supported by NSF grant DMS-0649473 and a grant from the Macarthur Foundation.

\section{Proof of Theorem \ref{main}}\label{mainproof}

We first establish uniqueness.  Suppose we have two semi-Strichartz class solutions $u, u' \in L^\infty_t L^2(I \times \R^2)$ to \eqref{nls}.  Let $J$ be the connected component of $\{ t \in I: u(t) = u(t')\}$ that contains $t_0$.  Since $u, u'$ are weak solutions, we see from Remark \ref{weak2} that $J$ is closed.  From the uniqueness component of Proposition \ref{strichprop}, and Definition \ref{semidef}, we also see that $J$ is right-open in $I \cap [t_0,+\infty)$ (i.e. for each $t \in J \cap [t_0,+\infty)$ there exists $\eps > 0$ such that $I \cap [t,t+\eps) \subset J$) and left-open in $I \cap (-\infty,t_0]$; by connectedness we conclude that $J=I$, establishing uniqueness.

Now we establish global existence.  It suffices to establish a semi-Strichartz class solution on $[t_0,+\infty)$, as by time reversal symmetry we may then obtain a semi-Strichartz class solution and $(-\infty,t_0]$, and glue them together to obtain the desired global solution on $\R$.

Let $J$ denote the set of all times $T \in [t_0,+\infty)$ for which there exists a semi-Strichartz class solution $u$ on $[t_0,T]$ with $M(u(t))$ monotone non-increasing on $[t_0,T]$, thus $J$ is a connected subset of $[t_0,+\infty)$ containing $t_0$.  By the existence and mass conservation component of Proposition \ref{strichprop}, we see that $J$ is right-open.  Now we establish that $J$ is closed.  If $t_n$ is a sequence of times in $J$ increasing to a limit $t_*$, then by gluing together all the associated semi-Strichartz class solutions (using uniqueness) we obtain a semi-Strichartz solution $u$ on $[t_0,t_*)$ with $M(u(t))$ monotone non-increasing on $[t_0,t_*)$; in particular $u$ lies in $L^\infty_t L^2_x( [t_0,t_*) \times \R^d)$, and $F(u)$ lies in $L^\infty_t L^{2d/(d+4)}_x([t_0,t_*) \times \R^d)$.  From this we see that the right-hand side of \eqref{integral} is continuous all the way up to $t_*$ in the space of tempered distributions, and so we can extend $u$ as a weak solution to $[t_0,t_*]$.  This is still a semi-Strichartz solution, and by Fatou's lemma we see that $M(u(t))$ is still non-decreasing on $[t_0,t_*]$, and so $t_* \in J$, thus establishing that $J$ is closed.  By connectedness we conclude that $J = [t_0,+\infty)$, and so we can obtain semi-Strichartz class solutions on $[t_0,T]$ for any $t_0 \leq T < \infty$.  Gluing these solutions together we obtain the desired solution on $[t_0,+\infty)$, establishing global existence.

The above argument has also established monotonicity of mass.  Whenever $u$ is a Strichartz class solution in a neighbourhood of a time $t_1$, it follows from Proposition \ref{strichprop} that mass is constant near $t_1$, so the only remaining task is to show that whenever mass is continuous at a time $t_1$, then $u$ is a Strichartz class solution in a neighbourhood of $t$.

The claim is obvious for $t_1=t_0$, so without loss of generality we may take $t_1 > t_0$.  By hypothesis, $M(u(t))$ converges to $M(u(t_1))$ as $t \to t_1$.  By \eqref{cosine}, we conclude that $u(t)$ converges strongly to $u(t_1)$ in $L^2_x(\R^d)$.

Let $\eps > 0$ be a small number.  By the endpoint Strichartz estimate \eqref{endpoint-strichartz}, we have 
$$ \| e^{i(t-t_1)\Delta} u(t_1) \|_{L^2_t L^{2d/(d-2)}_x(\R \times \R^d)} < \infty$$
so by the monotone convergence theorem we have
$$ \| e^{i(t-t_1)\Delta} u(t_1) \|_{L^2_t L^{2d/(d-2)}_x(I \times \R^d)} < \eps$$
when $I$ is a sufficiently small neighbourhood of $t_1$.

Fix $I$.  Let $t_2$ converge to $t_1$, then by the previous discussion $u(t_2)$ converges strongly to $u(t_1)$ in $L^2_x(\R^d)$.  By the continuity (and unitarity) of the Schr\"odinger propagator, this implies that $e^{i(t_2-t_1)\Delta} u(t_2)$ converges to $u(t_1)$.  Applying the endpoint Strichartz estimate \eqref{endpoint-strichartz}, we conclude that $e^{i(\cdot-t_2)\Delta} u(t_2)$ converges in $L^2_t L^{2d/(d-2)}_x(I \times \R^d)$ to $e^{i(\cdot-t_1)\Delta} u(t_1)$.  In particular, we have
$$ \| e^{i(t-t_2)\Delta} u(t_2) \|_{L^2_t L^{2d/(d-2)}_x(I \times \R^d)} < \eps$$
for all $t_2$ sufficiently close to $t_1$.  On the other hand, we have $M(u(t_2)) \leq M(u_0)$.  Thus if $\eps$ is chosen to be sufficiently small depending on $M(u_0)$, we may apply the standard Picard iteration argument based on the endpoint Strichartz estimate \eqref{endpoint-strichartz} and construct a Strichartz-class solution to NLS on $I$ which equals $u(t_2)$ at $t_2$.  Applying this with $t_2$ slightly smaller than $t_1$ and using the uniqueness of semi-Strichartz class solutions, we see that $u$ is equal to this Strichartz-class solution on $I$, and the claim follows.

\section{Proof of Theorem \ref{quant}}\label{quant-sec}

It is convenient here to use the original non-endpoint Strichartz estimate
\begin{equation}\label{nonendpoint-strichartz}
\| u \|_{L^{2(d+2)/d}_{t,x}(I \times \R^d)} +
\| u \|_{C^0_t L^2_x(I \times \R^d)} \lesssim_d \|u(t_0)\|_{L^2_x(\R^d)} + \| i u_t + \Delta u \|_{L^2_t L^{2(d+2)/(d+4)}_x(I \times \R^d)}
\end{equation}
of Strichartz \cite{strichartz:restrictionquadratic}.

Let $\eps_d > 0$ be chosen later, and let $u$ be a semi-Strichartz class solution.  Suppose for contradiction that we had a jump discontinuity at some time $t_1$ of jump less than $\eps_d$.  As before we may assume without loss of generality that $t_1 > t_0$.

Let $t$ approach $t_1$ from below, then $M(u(t)) - M(u(t_1))$ converges to a limit less than $\eps_d$.  By \eqref{cosine}, we conclude that $\|u(t)-u(t_1)\|_{L^2_x(\R^d)}$ converges to a limit less than $\eps_d$.

By \eqref{nonendpoint-strichartz} and monotone convergence as before, we can find a small neighbourhood $I$ of $t_1$ such that
$$ \| e^{i(t-t_1)\Delta} u(t_1) \|_{L^{2(d+2)/d}_{t,x}(I \times \R^d)} < \eps_d^{1/2}.$$
If we let $t_2$ approach $t_1$ from below, then for $t_2$ sufficiently close to $t_1$ we thus have
$$ \| e^{i(t-t_2)\Delta} u(t_1) \|_{L^{2(d+2)/d}_{t,x}(I \times \R^d)} < \eps_d^{1/2}.$$
On the other hand, from \eqref{nonendpoint-strichartz} we have (for $t_2$ sufficiently close to $t_1$) that
$$ \| e^{i(t-t_2)\Delta} (u(t_1)-u(t_2)) \|_{L^{2(d+2)/d}_{t,x}(I \times \R^d)} \lesssim_d \|u(t_2) - u(t_1) \|_{L^2_x(\R^d)} \lesssim \eps_d^{1/2}$$
and thus by the triangle inequality
$$ \| e^{i(t-t_2)\Delta} u(t_2) \|_{L^{2(d+2)/d}_{t,x}(I \times \R^d)} \lesssim_d \eps_d^{1/2}.$$
If $\eps_d$ is sufficiently small depending on $d$, we can then perform a Picard iteration, using \eqref{nonendpoint-strichartz} to control the nonlinear portion $u(t) - e^{i(t-t_2)\Delta} u(t_2)$ of the solution, to construct a solution in the class $C^0_t L^2_x \cap L^{2(d+2)/d}_{t,x}(I \times \R^d)$ that equals $u(t_2)$ on $t_2$.  Applying Strichartz estimates once more, we see that this solution is a Strichartz solution on $I$.  By uniqueness of semi-Strichartz solutions, we conclude that $u$ is a Strichartz solution on $I$ and thus has no jump discontinuity at $t_1$, a contradiction.

Now we handle the spherically symmetric case.  We will need the following result from \cite{kvz}:

\begin{theorem}[Scattering below the ground state]\label{scat}  Let $d \geq 3$.  Then for every $0 < m < M(Q)$ there exists a quantity $A(m) < \infty$ such that whenever $t_0 \in \R$ and $u_0 \in L^2_x(\R^d)$ with $M(u_0) \leq m$, then there exists a global Strichartz-class solution $u$ to \eqref{nls} with $\|u\|_{L^2_t L^{2d/(d-2)}_x(\R \times \R^d)} \leq A(m)$.
\end{theorem}

\begin{proof} See \cite[Theorem 1.5]{kvz}.  
\end{proof}

Now suppose for contradiction that we have a global semi-Strichartz class solution from spherically symmetric initial data $u_0$ which has a mass jump discontinuity of less than $M(Q)$ at some time $t_1$; we can assume $t_1 > t_0$ as before.

Since $u_0$ is spherically symmetric, we see from rotation invariance and uniqueness that $u$ is spherically symmetric.  By arguing as before, we see that as $t_2$ approaches $t_1$ from below, $M(u(t_2)-u(t_1))$ converges to a limit less than $M(Q)$.  In particular this limit is less than $m$ for some $0 < m < M(Q)$.

Let $\eps > 0$ be a small number depending on $m$ and $M(u_0)$ to be chosen later.  By endpoint Strichartz \eqref{endpoint-strichartz} and monotone convergence as before, we can find a small neighbourhood $I$ of $t_1$ such that
$$ \| e^{i(t-t_1)\Delta} u(t_1) \|_{L^2_t L^{2d/(d-2)}_x(I \times \R^d)} < \eps$$
and thus for $t_2$ sufficiently close to $t_1$
\begin{equation}\label{utd}
\| e^{i(t-t_2)\Delta} u(t_1) \|_{L^2_t L^{2d/(d-2)}_x(I \times \R^d)} < \eps.
\end{equation}
On the other hand, we also have
$$ M(u(t_2)-u(t_1)) < m$$
for $t_2$ sufficiently close to (and below) $t_1$.  By Theorem \ref{scat}, we may thus find a Strichartz-class solution $v$ on $I$ of mass at most $m$ with $v(t_2) = u(t_2) - u(t_1)$ and
$$ \| v \|_{L^2_t L^{2d/(d-2)}_x(I \times \R^d)} \leq A(m).$$
From this and \eqref{utd} and standard perturbation theory (see \cite[Lemma 3.1]{compact}), we may thus find a Strichartz-class solution on $I$ which equals $u(t_2)$ at $t_2$.  Arguing as before we conclude that $u$ has no jump discontinuity at $t_1$, a contradiction.

\begin{remark}  It is conjectured that the spherical symmetry assumption can be removed from Theorem \ref{scat}.  If this conjecture is true, then it is clear that one can take $\eps_d = M(Q)$ in the non-spherically-symmetric case of Theorem \ref{quant} as well.
\end{remark}

\section{A smoothing effect for spherically symmetric weak solutions}\label{weak-sec}

In this section we establish a preliminary smoothing effect for spherically symmetric weak solutions that will be needed to prove Theorems \ref{uncond},\ref{semis}.  More precisely, we show

\begin{theorem}[Smoothing effect]\label{smthm}  Let $d \geq 4$, let $I$ be a compact interval, and let $u \in L^\infty_t L^2_x(I \times \R^d)$ be a spherically symmetric weak solution to NLS with $M(u(t)) \leq m$ for all $t \in I$.  Then for every $R > 0$ one has the bound
\begin{equation}\label{preli}
\| u \|_{L^2_t L^{2d/(d-2)}_x( I \times (\R^d \backslash B(0,R)) )} \lesssim_{I,m,d} R^{-1} + 1,
\end{equation}
where $B(0,R)$ is the ball of radius $R$ centred at the origin.
\end{theorem}

\begin{remark}\label{smrem}  Theorem \ref{smthm} asserts that a spherically symmetric weak solution behaves like a Strichartz-class solution away from the spatial origin.  The $R^{-1}$ term on the right-hand side is sharp, as can be seen by considering a rescaled stationary soliton $u(t,x) = R^{-d/2} e^{it/R^2} Q(x/R)$, where $Q$ is a non-trivial spherically symmetric solution to \eqref{gse}.
\end{remark}

We shall prove this theorem using the method of in/out projections, as used in \cite{tao:radialfocus}, \cite{ktv}, \cite{kvz}.  We first recall some Littlewood-Paley notation.

Let $\varphi(\xi)$ be a radial bump function supported in the ball $\{ \xi \in \R^d: |\xi| \leq \tfrac {11}{10} \}$ and equal to $1$
on the ball $\{ \xi \in \R^d: |\xi| \leq 1 \}$.  For each number $N > 0$, we define the Fourier multipliers
\begin{align*}
\widehat{P_{\leq N} f}(\xi) &:= \varphi(\xi/N) \hat f(\xi)\\
\widehat{P_{> N} f}(\xi) &:= (1 - \varphi(\xi/N)) \hat f(\xi)\\
\widehat{P_N f}(\xi) &:= \psi(\xi/N)\hat f(\xi) := (\varphi(\xi/N) - \varphi(2\xi/N)) \hat f(\xi).
\end{align*}
We similarly define $P_{<N}$ and $P_{\geq N}$.   All sums over $N$ will be over integer powers of two unless otherwise stated.

We now subdivide the Littlewood-Paley projections $P_N$ on the spherically symmetric space $L^2(\R^d)_{\rad}$ into two components, an ``outgoing projection'' $P_+ P_N$ and ``incoming projection'' $P_- P_N$, as described in the following lemma:

\begin{proposition}[In/out decomposition]\label{inout}  Let $d \geq 1$.  Then there exist bounded linear operators $P_+, P_-: L^2(\R^d) \to L^2(\R^d)$ with the following properties:
\begin{itemize}
\item[(i)] $P_+, P_-$ extend to bounded linear operators on $L^p(\R^d)$ to $L^p(\R^d)$ for every $1< p < \infty$.
\item[(ii)] $P_+ + P_-$ is the orthogonal linear projection from $L^2(\R^d)$ to $L^2(\R^d)_{\rad}$.
\item[(iii)] For any $N > 0$, $|x| \gtrsim N^{-1}$, $t \gtrsim N^{-2}$, and choice of sign $\pm$, the integral kernel\footnote{The integral kernel $T(x,y)$ of a linear operator $T$ is the function for which $Tf(x) = \int_{\R^d} K(x,y) f(y)\ dy$ for all test functions $f$.} $[P_\pm P_N e^{\mp it\Delta}](x,y)$ obeys the estimate
$$ |[P_\pm P_N e^{\mp it\Delta}](x,y)| \lesssim_d (|x| |y|)^{-(d-1)/2} |t|^{-1/2}$$
when $|y|-|x| \sim N |t|$, and
$$ |[P_\pm P_N e^{\mp it\Delta}](x,y)| \lesssim_{d,m} \frac{N^d}{(N|x|)^{(d-1)/2} \langle N|y| \rangle^{(d-1)/2}} \langle N^2 t + N|x| - N|y| \rangle^{-m} $$
for any $m \geq 0$ otherwise.
\item[(iii)] For any $N > 0$, $|x| \gtrsim N^{-1}$, $|t| \lesssim N^{-2}$, and choice of sign $\pm$, we have
$$ |[P_\pm P_N e^{\mp it\Delta}](x,y)| \lesssim_{d,m} \frac{N^d}{(N|x|)^{(d-1)/2} \langle N|y| \rangle^{(d-1)/2}} \langle N|x| - N|y| \rangle^{-m} $$
for any $m \geq 0$.
\end{itemize}
\end{proposition}

\begin{proof}  See \cite[Proposition 6.2]{ktv} (for the $d=2$ case) or \cite[Lemma 4.1, Lemma 4.2]{kvz} (for the higher $d$ case).
\end{proof}

\begin{remark} Heuristically, $P_- P_N e^{it\Delta}$ and $P_+ P_N e^{-it\Delta}$ for $t > 0$ both propagate away from the origin at speeds $\sim N$.  The decay $(|x| |y|)^{-(d-1)/2} |t|^{-1/2}$ is superior to the standard decay $|t|^{-d/2}$, which reflects the additional averaging away from the origin caused by the spherical symmetry.  (In the proof of \cite[Proposition 6.2]{ktv}, this additional averaging is captured using the standard asymptotics of Bessel and Hankel functions.)
\end{remark}

Now we prove Theorem \ref{smthm}.  Fix $d,I,u,m,R$; we allow implied constants to depend on $d,I,m$.  We may take $R$ to be a power of $2$. By the triangle inequality, we have
\begin{align*}
\| u \|_{L^2_t L^{2d/(d-2)}_x( I \times (\R^d \backslash B(0,R)) )} 
&\lesssim \| P_{\leq 1/R} u \|_{L^2_t L^{2d/(d-2)}_x(I \times \R^d)} \\
&\quad + \sum_{N > 1/R} \sum_\pm \| P_\pm P_N u \|_{L^2_t L^{2d/(d-2)}_x( I \times (\R^d \backslash B(0,R)) )}.
\end{align*}
For the first term, we use Bernstein's inequality to estimate
$$ \| P_{\leq 1/R} u(t) \|_{ L^{2d/(d-2)}_x( \R^d ) } \lesssim R^{-1} \| u(t) \|_{L^2_x(\R^d)} \lesssim R^{-1}$$
which is acceptable, so we turn to the latter terms.  For ease of notation we shall just deal with the the incoming terms $\pm = -$, as the outgoing terms $\pm = -$ terms are handled similarly (but using Duhamel backwards in time instead of forwards).

Write $I = [t_0,t_1]$, then by Duhamel's formula we have
$$ P_- P_N u(t) = P_- P_N e^{i(t-t_0)\Delta} u(t_0) - i \int_{t_0}^t P_- P_N e^{i(t-t')\Delta} F(u(t'))\ dt'.$$
The contribution of the linear term $P_- P_N e^{i(t-t_0)\Delta} u(t_0)$ is bounded by
\begin{align*}
\| \sum_{N > 1/R} P_- P_N e^{i(t-t_0)\Delta} u(t_0) \|_{L^2_t L^{2d/(d-2)}_x(I \times \R^d)} 
&\lesssim
\| \sum_{N > 1/R} P_N e^{i(t-t_0)\Delta} u(t_0) \|_{L^2_t L^{2d/(d-2)}_x(I \times \R^d)}  \\
&\lesssim \| e^{i(t-t_0)\Delta} u(t_0) \|_{L^2_t L^{2d/(d-2)}_x(I \times \R^d)}  \\
&\lesssim \| u(t_0) \|_{L^2_x(\R^d)} \\
&\lesssim 1
\end{align*}
thanks to Proposition \ref{inout}(i), the boundedness of the Littlewood-Paley projection $P_{>1/R}$, and the endpoint Strichartz estimate \eqref{endpoint-strichartz}.  Thus this contribution is acceptable, and it remains to show that
\begin{equation}\label{sumnr}
 \sum_{N > 1/R} \| \int_{t_0}^t P_- P_N e^{i(t-t')\Delta} F(u(t'))\ dt' \|_{L^2_t L^{2d/(d-2)}_x(I \times (\R^d \backslash B(0,R)) )} \lesssim R^{-1}.
\end{equation}
As we are allowed to let implied constants depend on $I$, it suffices to show that
$$ \int_{t_0}^t \| P_- P_N e^{i(t-t')\Delta} F(u(t'))\|_{L^{2d/(d-2)}_x(\R^d \backslash B(0,R))}\ dt'  \lesssim (NR)^{-c} R^{-1}$$
for some absolute constant $c>0$ and all $t \in I$ and $N > 1/R$.  By dyadic decomposition it suffices to show that
$$ \int_{t_0}^t \| P_- P_N e^{i(t-t')\Delta} F(u(t'))\ dt' \|_{L^{2d/(d-2)}_x(B(0,2^{m+1}R) \backslash B(0,2^m R))}\ dt'  \lesssim (2^m NR)^{-c} R^{-1}$$
for all $m \geq 0$.  Replacing $R$ by $2^m R$, we thus see that it suffices to show that
$$ \int_{t_0}^t \| P_- P_N e^{i(t-t')\Delta} F(u(t'))\ dt' \|_{L^{2d/(d-2)}_x(B(0,2R)) \backslash B(0,R)}\ dt' \lesssim (NR)^{-c} R^{-1}$$
whenever $R > 0$, $N > 1/R$, and $t \in I$.

From Proposition \ref{inout} we see that
$$ \| P_- P_N e^{it\Delta} f \|_{L^\infty_x(B(0,2R) \backslash B(0,R))} \lesssim (R (R+N|t|))^{-(d-1)/2} |t|^{-1/2} \|f\|_{L^1_x(\R^d)}$$
for $t \gtrsim N^{-2}$, and
$$ \| P_- P_N e^{it\Delta} f \|_{L^\infty_x(B(0,2R) \backslash B(0,R))} \lesssim R^{-(d-1)/2} N \|f\|_{L^1_x(\R^d)}$$
for $0 < t \lesssim N^{-2}$; we unify these two estimates as
$$ \| P_- P_N e^{it\Delta} f \|_{L^\infty_x(B(0,2R) \backslash B(0,R))} \lesssim (R (R+N|t|))^{-(d-1)/2} N \langle N^2 t\rangle^{-1/2} \|f\|_{L^1_x(\R^d)}$$
for $t > 0$.  On the other hand, as $P_-, P_N, e^{it\Delta}$ are bounded on $L^2$ we have
$$ \| P_- P_N e^{it\Delta} f \|_{L^2_x(B(0,2R) \backslash B(0,R))} \lesssim  \|f\|_{L^2_x(\R^d)}$$
and hence by interpolation
\begin{align*}
\| P_- P_N e^{it\Delta} &f \|_{L^{2d/(d-4)}_x(B(0,2R) \backslash B(0,R))} \\
&\lesssim [(R (R+N|t|))^{-(d-1)/2} N \langle N^2 t\rangle^{-1/2}]^{4/d} \|f\|_{L^{2d/(d+4)}_x(\R^d)}.
\end{align*}
Since
$$ \| F(u(t'))\|_{L^{2d/(d+4)}_x(\R^d)} \lesssim \|u(t') \|_{L^2_x(\R^d)}^{1+4/d} \lesssim 1$$
for all $t' \in I$, we thus have
\begin{align*}
\| P_- P_N e^{i(t-t')\Delta} F(u(t')) &\|_{L^{2d/(d-4)}_x(B(0,2R) \backslash B(0,R))} \\
&\lesssim [(R (R+N|t-t'|))^{-(d-1)/2} N \langle N^2 (t-t')\rangle^{-1/2}]^{4/d} 
\end{align*}
and hence by H\"older's inequality
\begin{align*}
\| P_- P_N e^{i(t-t')\Delta} F(u(t')) &\|_{L^{2d/(d-2)}_x(B(0,2R) \backslash B(0,R))} \\
&\lesssim R [(R (R+N|t-t'|))^{-(d-1)/2} N \langle N^2 (t-t')\rangle^{-1/2}]^{4/d}.
\end{align*}
We can thus bound the left-hand side of \eqref{sumnr} by
$$ \int_{-\infty}^t R [(R (R+N|t-t'|))^{-(d-1)/2} N \langle N^2 (t-t')\rangle^{-1/2}]^{4/d}\ dt'.$$
The dominant contribution of this integral occurs in the region when $|t-t'| \sim R/N$, and so we obtain a total contribution of  
$$ \lesssim R (R/N) ( R^{-(d-1)} (R/N)^{-1/2} )^{4/d} = R^{-1} (RN)^{-(d-2)/d}$$
which is acceptable.  This proves Theorem \ref{smthm}.

\begin{remark} One can improve the $1$ term on the right-hand side of \eqref{preli} to $R^{-c}$ for some $c > 0$, by using the improved Strichartz estimates in \cite{Shuanglin} that are available in the spherically symmetric case.  However, we will not need this improvement here.
\end{remark}

\section{Nearly continuous solutions are Strichartz class}

Theorem \ref{smthm} gives Strichartz norm control of a solution away from the spatial origin.  When the solution is sufficiently close in $L^\infty_t L^2_x$ to a Strichartz class solution, we can bootstrap Theorem \ref{smthm} to in fact obtain Strichartz control all the way up to the origin.  More precisely, we now show

\begin{theorem}[Strichartz class criterion]\label{strich}  Let $d \geq 5$, let $I$ be a compact interval, and let $u \in L^\infty_t L^2_x(I \times \R^d)$ be a spherically symmetric weak solution to NLS.  Suppose also that there exists a Strichartz-class solution $v \in C^0_t L^2_x \cap L^2_t L^{2d/(d-2)}_x(I \times \R^2)$ such that $\| u-v\|_{L^\infty_t L^2_x(\R^d)} \leq \eps$.  If $\eps$ is sufficiently small depending on $d$, then $u \in L^2_t L^{2d/(d-2)}_x(I \times \R^d)$.
\end{theorem}

\begin{remark} The theorem fails if $\eps$ is large, as one can see from the weak solution defined by \eqref{pc} for $t \neq 0$ and vanishing for $t=0$.  The arguments in fact give an effective upper bound for the $L^2_t L^{2d/(d-2)}_x$ norm of $u$ in terms of the corresponding norm of $v$.  Heuristically, the point is that when $u$ (or $u-v$) has small mass, then there are not enough nonlinear effects in play to support persistent mass concentration (as in the example in Remark \ref{smrem}) that would cause the $L^2_t L^{2d/(d-2)}$ norm to become large.
\end{remark}

We now prove Theorem \ref{strich}.  We fix $d,I,u,v,\eps$ and allow all implied constants to depend on $d$.  By shrinking the interval $I$ and using compactness we may assume that
\begin{equation}\label{vsm}
\|v\|_{L^2_t L^{2d/(d-2)}_x(I \times \R^2)} \leq \eps.
\end{equation}
We write $w := u-v$, thus 
\begin{equation}\label{wbound}
\|w\|_{L^\infty_t L^2_x(I \times \R^d)} \lesssim \eps
\end{equation}
and $w$ solves the difference equation
\begin{equation}\label{wdiff}
iw_t + \Delta w = F(v+w) - F(v) 
\end{equation}
in the integral sense.  From the fundamental theorem of calculus (or mean-value theorem) we have the elementary inequality
\begin{equation}\label{void}
F(v+w) - F(v) = O( |w| (|v|+|w|)^{4/d} ).
\end{equation}

For each integer $k$, let $c_k$ denote the quantity
\begin{equation}\label{cmdef}
c_k := \|w\|_{L^2_t L^{2d/(d-2)}_x(I \times (\R^d \backslash B(0,2^k)))}.
\end{equation}
From Theorem \ref{smthm} and the triangle inequality we have
\begin{equation}\label{cmb}
c_k \lesssim_I 2^{-k} + 1
\end{equation}
for all $k$.  To prove the theorem, it suffices by the monotone convergence theorem to show that $\sup_k c_k$ is finite.  For this we use the following inequality:

\begin{proposition}[Key inequality]\label{key}  Let $d \geq 4$. For every $k$ we have
$$ c_k \lesssim \eps + \eps^{4/d} \sum_{j \leq k} 2^{-\frac{d-2}{2}(k-j)} (c_j + c_j^{1-4/d}).$$
\end{proposition}

\begin{proof}  Fix $k$.  By the triangle inequality, we have
\begin{equation}\label{cm1}
\begin{split}
 c_k &\lesssim \| P_{\leq 2^{-k}} w(t) \|_{L^2_t L^{2d/(d-2)}_x(I \times \R^d)} \\
 &+ \sum_\pm 
\| P_\pm P_{>2^{-k}} w(t) \|_{L^2_t L^{2d/(d-2)}_x(I \times (\R^d \backslash B(0,2^k)))}.
\end{split}
\end{equation}
Consider the first term on the right-hand side.  By \eqref{wbound}, \eqref{wdiff}, \eqref{void}, \eqref{endpoint-strichartz} we have
$$ \| P_{\leq 2^{-k}} w(t) \|_{L^2_t L^{2d/(d-2)}_x(I \times \R^d)} 
\lesssim \eps + \| P_{\leq 2^{-k}} O( |w| (|v|+|w|)^{4/d} ) \|_{L^2_t L^{2d/(d+2)}_x(I \times \R^d)}$$
so to show that the contribution of this case is acceptable, it suffices to show that
$$ \| P_{\leq 2^{-k}} O( |w| (|v|+|w|)^{4/d} ) \|_{L^2_t L^{2d/(d+2)}_x(I \times \R^d)} \lesssim
\eps^{4/d} \sum_{j \leq k} 2^{-\frac{d-2}{2}(k-j)} (c_j + c_j^{4/d}).$$
By the triangle inequality, we can bound the left-hand side by
\begin{equation}\label{wow}
\begin{split}
& \| P_{\leq 2^{-k}} O( |w| 1_{\R^d \backslash B(0,2^{-k})} (|v|+|w|)^{4/d} ) \|_{L^2_t L^{2d/(d+2)}_x(I \times \R^d)}\\
&\quad + \sum_{j < k} \| P_{\leq 2^{-k}} O( |w| 1_{B(0,2^{-j+1}) \backslash B(0,2^{-j})} (|v|+|w|)^{4/d} ) \|_{L^2_t L^{2d/(d+2)}_x(I \times \R^d)}.
\end{split}
\end{equation}
For the first term, \eqref{wow} we discard the $P_{\leq 2^{-k}}$ projection and use H\"older's inequality to bound this by
\begin{align*}
&\lesssim \| w \|_{L^2_t L^{2d/(d+2)}_x(I \times (\R^d \backslash B(0,2^{-k})))}^{1-4/d}
\| v \|_{L^2_t L^{2d/(d+2)}_x(I \times \R^d)}^{4/d}
 \|w\|_{L^\infty_t L^2_x(I \times \R^d)}^{4/d}\\
&\quad +
\| w \|_{L^2_t L^{2d/(d+2)}_x(I \times (\R^d \backslash B(0,2^{-k})))} \|w\|_{L^\infty_t L^2_x(I \times \R^d)}^{4/d}
\end{align*}
which by \eqref{wbound}, \eqref{vsm}, \eqref{cmdef} is bounded by
$$ \lesssim \eps^{4/d} c_k^{1-4/d} + \eps^{4/d} c_k$$
which is acceptable.

For the second term of \eqref{wow}, we observe from the H\"older and Bernstein inequalities that
\begin{align*}
\| P_{\leq 2^{-k}} ( f 1_{B(0,2^{-j+1}) \backslash B(0,2^{-j})} ) \|_{L^{2d/(d+2)}_x(\R^d)} &\lesssim 
2^{-\frac{d-2}{2} k} \| f 1_{B(0,2^{-j+1}) \backslash B(0,2^{-j})} \|_{L^1_x(\R^d)} \\
&\lesssim 2^{-\frac{d-2}{2}(k-j)} \| f \|_{L^{2d/(d+2)}_x(B(0,2^{-j+1}) \backslash B(0,2^{-j}))}
\end{align*}
for any $f$.  Using this inequality and arguing as before, we see that the second term of \eqref{wow} is bounded by
$$ \lesssim \sum_{j<k}2^{-\frac{d-2}{2}(k-j)} (\eps^{4/d} c_j^{1-4/d} + \eps^{4/d} c_j)$$
which is acceptable.

Since we have dealt with the first term of \eqref{cm1}, it now suffices by the triangle inequality to show that
$$
\| P_\pm P_{> 2^{-k}} w(t) \|_{L^2_t L^{2d/(d-2)}_x(I \times (\R^d \backslash B(0,2^m)))}
\lesssim \eps + \eps^{4/d} \sum_{j \leq k} 2^{-\frac{d-2}{2}(k-j)} (c_j + c_j^{1-4/d})$$
for either choice of sign $\pm$.  We shall just do this for the incoming case $\pm = -$: the outgoing case $\pm=+$ is similar but requires one to apply Duhamel's formula backwards in time.

Write $I = [t_0,t_1]$.  By Duhamel's formula and \eqref{wdiff}, we have
$$ P_- P_{> 2^{-k}} w(t) = P_- P_{>2^{-k}} e^{i(t-t_0)\Delta} w(t_0) - i \int_{t_0}^t P_- P_{>2^{-k}} e^{i(t-t')\Delta} (F(v+w)-F(v))(t')\ dt'.$$
The contribution of the first term is $O(\eps)$ by Proposition \ref{inout}(i), \eqref{endpoint-strichartz}, and \eqref{wbound}, so it suffices to show that
\begin{align*}
\| \int_{t_0}^t P_- P_{>2^{-k}} &e^{i(t-t')\Delta} (F(v+w)-F(v))(t')\ dt' \|_{L^2_t L^{2d/(d-2)}_x(I \times (\R^d \backslash B(0,2^k)))} \\
&\lesssim \eps^{4/d} \sum_{j \leq k} 2^{-\frac{d-2}{2}(k-j)} (c_j + c_j^{1-4/d}).
\end{align*}
We split
$$ F(v+w)-F(v) = (F(v+w)-F(v)) 1_{\R^d \backslash B(0,2^{k-1})} + \sum_{j < k-1} (F(v+w)-F(v)) 1_{B(0,2^{j+1}) \backslash B(0,2^j)}.$$
The contribution of the first term can be estimated using Proposition \ref{inout}(i), \eqref{endpoint-strichartz}, \eqref{void} to be
$$ \lesssim \| |w| (|v|^{4/d} + |w|^{4/d}) \|_{L^2_t L^{2d/(d+2)}_x(I \times (\R^d \backslash B(0,2^{k-1})))}.$$
By a slight modification of the calculation used to bound the first term of \eqref{wow}, we can control this expression by
$$ \lesssim \eps^{4/d} c_{k-1}^{1-4/d} + \eps^{4/d} c_{k-1}$$
and so by the triangle inequality it suffices to show that
\begin{equation}\label{sumnd}
\begin{split}
 & \sum_{N > 2^{-k}} \| \int_{t_0}^t P_- P_N e^{i(t-t')\Delta} [(F(v+w)-F(v))(t') 1_{B(0,2^{j+1}) \backslash B(0,2^j)}]\ dt' \|_{L^2_t L^{2d/(d-2)}_x(I \times (\R^d \backslash B(0,2^k)))} \\
 &\quad \lesssim \eps^{4/d} 2^{-\frac{d-2}{2}(k-j)} (c_j + c_j^{1-4/d})
 \end{split}
\end{equation}
for each $j < k-1$.

Fix $j$.  By Proposition \ref{inout}(ii), (iii), the integral kernel $(P_- P_N e^{i(t-t')\Delta})(x,y)$ for $x \in \R^d \backslash B(0,2^k)$, $t' < t$, $N > 2^{-k}$, and $y \in B(0,2^{j+1}) \backslash B(0,2^j)$ obeys the bounds
\begin{align*}
|[P_- P_N e^{i(t-t')\Delta}](x,y)| &\lesssim \frac{N^d}{(N|x|)^{(d-1)/2} \langle 2^j N \rangle^{(d-1)/2}} \langle N^2 (t-t') + N|x| \rangle^{-100d} \\
&\lesssim N^d (N|x|)^{-50d} \langle N^2 (t-t') \rangle^{-50d}
\end{align*}
(say).  From this we obtain the pointwise bound
$$ | P_- P_N e^{i(t-t')\Delta} ( f 1_{B(0,2^{j+1}) \backslash B(0,2^j)} )(x)| \lesssim N^d (N|x|)^{-50d} \langle N^2 (t-t') \rangle^{-50d} \|f\|_{L^1_x(B(0,2^{j+1}) \backslash B(0,2^j))} $$
for $x \in \R^d \backslash B(0,2^k)$ and any $f$, which by H\"older's inequality implies the bounds
\begin{align*}
\| &P_- P_N e^{i(t-t')\Delta} ( f 1_{B(0,2^{j+1}) \backslash B(0,2^j)} ) \|_{L^{2d/(d-2)}_x(\R^d \backslash B(0,2^k))}\\
&\quad \lesssim 2^{\frac{d-2}{2} k} 2^{\frac{d-2}{2} j} 
N^d (2^k N)^{-50d} \langle N^2 (t-t') \rangle^{-50d} \|f\|_{L^{2d/(d+2)}_x(B(0,2^{j+1}) \backslash B(0,2^j))}.
\end{align*}
By Young's inequality we conclude that the left-hand side of \eqref{sumnd} is bounded by
$$ \lesssim \sum_{N > 2^{-k}} 2^{\frac{d-2}{2} k} 2^{\frac{d-2}{2} j} N^{d-2} (2^k N)^{-50d} 
\| F(v+w)-F(v) \|_{L^2_t L^{2d/(d+2)}_x(I \times B(0,2^{j+1}) \backslash B(0,2^j))}.$$
Modifying the computation used to bound the first term of \eqref{wow}, this expression can be controlled by
$$ \lesssim \sum_{N > 2^{-k}} 2^{\frac{d-2}{2} k} 2^{\frac{d-2}{2} j} N^{d-2} (2^k N)^{-50d} 
( \eps^{4/d} c_j^{1-4/d} + \eps^{4/d} c_j )$$
and on performing the summation in $N$ one obtains the claim \eqref{sumnd}, and Proposition \ref{key} follows.
\end{proof}

From Proposition \ref{key} (and using the hypothesis $d \geq 5$ to make the decay $2^{-\frac{d-2}{2}(k-j)}$ faster tan the blowup of $2^{-j}$), we see that if we have any bound of the form
$$ c_k \leq A + B 2^{-k}$$
for all $k$ and some $A,B > 0$, then (if $\eps$ is sufficiently small, and $A$ is sufficiently large depending on $\eps$), one can conclude a bound of the form
$$ c_k \leq A + \frac{1}{2} B 2^{-k}$$
for all $k$.  Iterating this and taking limits, we conclude that
$$ c_k \leq A$$
for all $k$.  Applying this argument starting from \eqref{cmb} we conclude that $c_k \lesssim 1$ for all $k$, as desired, and Theorem \ref{strich} follows.

\section{Proofs of theorems}

With Theorem \ref{strich} in hand, it is now an easy matter to establish Theorems \ref{uncond} and \ref{semis}.

\subsection{Proof of Theorem \ref{uncond}}\label{uncond-sec}

It is clear that (i) implies (ii), and that (iii) implies (iv) implies (v).  From Proposition \ref{strichprop} we also see that (ii) implies (iii).  So the only remaining task is to show that (v) implies (i).  It suffices to do this locally, i.e. to show that for any time $t$ for which (v) holds, that $u$ is a Strichartz class solution in some neighbourhood of $t$ in $I$.

By the hypothesis (v), one can find a connected neighbourhood $J$ of $t$ in $I$ such that $M(u(t')) \leq M(u(t)) + \eps_d$ for all $t' \in J$.  By \eqref{cosine} (and shrinking $J$ if necessary) we conclude that $\|u(t')-u(t)\|_{L^2_x(\R^d)}^2 \leq 2 \eps_d$ (say) for all $t' \in J$.

By shrinking $J$ some more, we may apply Proposition \ref{strichprop} to find a Strichartz class solution $v \in C^0_t L^2_x \cap L^2_t L^{2d/(d-2)}_x(J \times \R^d)$ on $J$ with $v(t)=u(t)$.  Since $v$ is a strong solution, by shrinking $J$ some more we may assume that $\|v(t')-v(t) \|_{L^2_x(\R^d)} \leq \eps_d^{1/2}$ for all $t' \in J$.  By the triangle inequality we thus see that $\|u-v\|_{L^\infty_t L^2_x(J \times \R^d)} \lesssim \eps_d^{1/2}$.  Applying Theorem \ref{strich} and taking $\eps_d$ sufficiently small, we conclude that $u$ is a Strichartz class solution on $J$ as required, and Theorem \ref{uncond} follows.

\subsection{Proof of Theorem \ref{semis}}\label{semis-sec}

It is clear that (i) implies (ii) and that (iii) implies (iv) implies (v).  From Corollary \ref{quantcor} we know that (i) implies (iii), while from Proposition \ref{strichprop}(iii) and Definition \ref{semidef} we see that (ii) implies (iv).  Thus, as before, the only remaining task is to show that (v) implies (i).  Again, it suffices to establish the local claim that if $t \geq t_0$ is such that (v) holds, then $u$ is in the Strichartz class for some $[t,t+\eps) \cap I$, and similarly for $t \leq t_0$ and $(t-\eps,t] \cap I$.  But this follows by a routine modification of the arguments in Section \ref{uncond-sec}.

\end{document}